\documentclass[12pt,letterpaper]{amsart}
\usepackage{amsmath, amsthm, amssymb, xspace, mathrsfs}
\usepackage{geometry}
\usepackage[breaklinks=true]{hyperref}
\usepackage{amsmath,amscd}

\theoremstyle{plain}
\newtheorem{theorem}{Theorem}[section]
\newtheorem{lemma}{Lemma}[section]
\newtheorem{prop}{Proposition}[section]
\newtheorem{cor}{Corollary}[section]
\newtheorem{Question}{Question}
\theoremstyle{definition}

\newtheorem{remark}{Remark}[section]
\newtheorem*{acknowledgement}{Acknowledgements}

\newcommand{\disk}{\mathbb{D}}

\newcommand{\comment}[1]{}
\DeclareMathOperator{\coker}{\ensuremath{coker}}

\begin{document}
\title{Invertibility of Toeplitz operators with polyanalytic symbols}
\author{Akaki Tikaradze}
\email{ tikar06@gmail.com}
\address{University of Toledo, Department of Mathematics \& Statistics, 
Toledo, OH 43606, USA}
\begin{abstract}
For a class of continuous functions  including complex polynomials in $z,\bar{z},$ we show that
the corresponding Toeplitz operator on the Bergman space of the unit disc
can be expressed as a quotient of certain differential operators with holomorphic coefficients. 
This enables us to obtain several nontrivial operator theoretic results about such Toeplitz operators, 
including a new criterion for invertibility of a Toeplitz operator for a class of harmonic symbols.

\end{abstract}

\maketitle

\section{Introduction}

Throughout $\mathbb{D}=\lbrace z\in \mathbb{C}:|z|<1\rbrace$ will denote the unit disc and
$A^2(\disk)$ will denote the Bergman space of square integrable holomorphic functions on $\mathbb{D}$
with respect to the normalized Lebesgue measure.
Also, by $H(\mathbb{D})$ will denote the Frechet space of all holomorphic functions on $\mathbb{D}$, while $H(\overline{\mathbb{D}})$
denotes the set of all holomorphic functions defined on a neighbourhood of $\overline{\mathbb{D}}.$
 given a bounded measurable function $g\in L^{\infty}(\mathbb{D})$, we will denote by $T_g:A^2(\mathbb{D})\to A^2(\mathbb{D})$
the corresponding Toeplitz operator. Recall its definition: $T_g(f)=P(gf),$ where $P:L^2(\mathbb{D})\to A^2(\mathbb{D})$
denotes the orthogonal projection.
Let $\phi\in H(\overline{\mathbb{D}})[\bar{z}]$ be a polyanalytic function on $\overline{\mathbb{D}}.$ We are interested in the question of invertibility of
the corresponding Toeplitz operator $T_{\phi}:A^2(\mathbb{D})\to A^2(\mathbb{D}).$
More generally we are interested in determining dimensions of its kernel and cokernel.
Similar problem in the setting of the Hardy space $H^2(\mathbb{D})$ is well understood.
Indeed, recall that a well-known theorem of Coburn asserts that given any $g\in L^{\infty}(\mathbb{D})$, then the corresponding
Toeplitz operator $T_g:H^2(\mathbb{D})\to H^2(\mathbb{D})$ is either injective, or its conjugate $T_{\bar{g}}=T_g^*$ is injective.
We will recall the following refinement of the Coburn's theorem for symbols that are continuous up to the boundary of
$\mathbb{D}$ (such symbols will be the object of our main interest.)

\begin{theorem}\label{Douglas}(\cite{D}). Let $g\in C(\bar{\mathbb{D}}).$ 
 then $T_g:H^2(\mathbb{D})\to H^2(\mathbb{D})$ is a Fredholm operator
if and only if  $g$ does not vanish on $\partial \mathbb{D}.$
Let $n$ be the
the winding number of $g(\partial \mathbb{D})$ around $0.$
If  $n>0,$ then $T_g:H^2(\mathbb{D})\to H^2(\mathbb{D})$ is onto with $n$-dimensional kernel. If $n<0$ then
$T_g$ is injective with $n$-dimensional cokernel. Finally $T_g$ is invertible if $n=0.$
\end{theorem}

We recall the following  well-known partial analogue of this statement in the Bergman space
setting (see for example [\cite{SZ1}, Theorem 24].) It provides the full description of all Fredholm
operators of the form $T_g, g\in g\in C(\bar{\mathbb{D}}).$

\begin{lemma}\label{Index}
If $g\in C(\bar{\mathbb{D}})$, then $T_g:A^2(\mathbb{D})\to A^2(\mathbb{D})$ is a Fredholm operator
if and only if  $g$ does not vanish on $\partial \mathbb{D},$ in this case its index
 equals to minus of the
the winding number of $g(\partial \mathbb{D})$ around $0.$

\end{lemma}

The full analogue of Theorem \ref{Douglas} in the Bergman space setting fails even for harmonic functions: 
Sundberg and Zheng \cite{SZ}
constructed an example $g=\bar{z}+\phi, \phi\in H(\overline{\mathbb{D}}$) such that
$T_g$ is not invertible while the winding number of $g(\partial \mathbb{D})$ around 0
is 0.

Determining dimensions of $\ker(T_g),\coker(T_g)$ for general classes of harmonic symbols $g$ (in the Bergman setting)
 is a fundamental problem,
full solution to which in seems to be out of reach at the moment.
We will also recall that for real harmonic function $h, T_h$ is invertible if $h(\mathbb{D}$) is bounded away from zero
by Mcdonald-Sundberg \cite{MS}.

Let us start by recalling results of Sundberg and Zheng \cite{SZ} in more detail. They made a crucial observation that
$$T_{\bar{z}}(f)=\frac{1}{z^2}\int_0^z wf'(w)dw.$$
Based on this it is easy to deduce that given $g=\bar{z}+\phi$ then $T_{g}(f)=0$ if and only
if $f$ satisfies the following first order differential equation
$$(1+z\phi(z))f'(z)=-(2\phi(z)+z\phi'(z))f(z).$$
Thus $T_g$ is invertible iff 
$$res_w\frac{2\phi(z)+z\phi'(z)}{1+z\phi(z)}\in \mathbb{Z}_{\leq 0}.$$
This observation led Sundberg and Zheng to a construction of a rational function $\phi(z)$ with poles outside $\overline{\mathbb{D}},$
such that $g=\bar{z}+\phi(z)$ has the property that $T_g$ is a Fredholm operator of index 0, but $\ker(T_g)$ (hence $\coker(T_g))$
is nontrivial. Moreover 0 is an isolated element of the spectrum of $T_g$  [\cite{SZ}, Theorem 2.3, Lemma  2.2].

 We will use the following notation/convention to state our main results.
 Given an $n$-th order polyanalytic function
$\phi(z)=\sum_{i=0}^n a_i(z)\bar{z}^i, a_i\in H(\mathbb{D}),$ we will define the following holomorphic function
$\tilde{g}$ as follows:
$$\tilde{g}(z)=\sum a_i(z)z^{n-i}.$$
The crucial relation between $\phi$ and $\tilde{\phi}$ is that $\phi(z)=z^{-n}\tilde{\phi}(z)|_{\partial \mathbb{D}}.$
Therefore it follows from the argument principle that (assuming $\phi\in C(\overline{\mathbb{D}})$
 if the winding number of $\phi(\partial \mathbb{D})$ around 0 is $m,$ 
then $\tilde{\phi}$ has $n-m$ zeros
on $\mathbb{D}.$

Now we will state  our main results. Given an $n$-th order polyanalytic
function $\phi(z)=\sum_{i=0}^n a_i(z)\bar{z}^i, a_i(z)\in H(\mathbb{D}),$
 we will define the following $n$-th order differential operator 

$$D_{\phi}=\prod_{i=2}^{n+1}(zD+i)a_0(z)+\sum_{i=1}^n\prod_{k=i+2}^{n+1}(zD+k)D^ia_i(z), D=\frac{\partial}{\partial z}.$$

The following is the key result.

\begin{lemma}\label{key}
Let $f, g\in A^2(\mathbb{D}).$ Put $\phi(z)=\sum_{i=1}^n a_i(z)\bar{z}^i, a_i(z)\in H^{\infty}(\mathbb{D}).$
Then $T_{\phi}(f)=g$ if and only if 
$\prod_{i=2}^{n+1}(zD+i)(g)=D_{\phi}(f).$ In particular $T_{\phi}(f)=0$ if and only if $D_{\phi}(f)=0.$

\end{lemma}

This result allows us to transfer the problems about the kernel of Toeplitz operators
with polyanalytic symbols (in particular questions about their invertibility) to the problems about existence of
solutions of holomorphic ordinary differential equations. 
The key for proving the above result will be to explicitly realize Toeplitz operators with polyanalytic symbols
as a fraction of differential operators with analytic coefficients. This enables us to embed the algebra
generated of Toeplitz operators with polyanalytic symbols into a skew field of analytic differential operators
on $D$. As an immediate corollary we obtain the following.

\begin{prop}
The algebra generated by all Toeplitz operators with polyanalytic symbols
has no zero divisors.
\end{prop}

 The problem of analysing $T_{\phi}$ for $n$-th order polyanalytic functions
can be naturally broken up into studying $\phi$ for which $\tilde{\phi}$ has $m$ zeros in $\mathbb{D}$ for each
nonnegative integer value of $m.$ More specifically, given $w_1, \cdots, w_m\in D$ (not necessarily distinct)
we would like to analyse $\dim \ker T_{\phi}$ for $n$-th order polyanalytic $\phi$ such that zeroes of $\tilde{\phi}$ in $\overline{\mathbb{D}}$
are precisely $w_1,\cdots w_m.$ 
The case of $n=m$ is perhaps the most interesting since this is the case for index $0$ $T_{\phi}.$
A deep connection between operator theoretic properties of $T_{\phi}$ and function theoretic properties
of $\tilde{\phi}$ is highlighted by the fact (to be proven below) that equation $T_{\phi}(y)=0$ is equivalent to
the differential equation $D_{\phi}y=0$ whose singularities are precisely
at zeros of $\tilde{\phi}.$
Our main results provide full answer to this classification question
for $m\leq 1$, as well as for the case of $m=n$ and  $w_1=w_2=\cdots=w_m.$

\begin{theorem}\label{res}
Let $\phi\in H^{\infty}(\mathbb{D})[\bar{z}]$ be an $n$-th order polyanalytic function.
Then the kernel of $T_{\phi}$ is at most $n$-dimensional. 
If $\phi\in H^{\infty}(\overline{\mathbb{D}})[\bar{z}]$ and 
 $\tilde{\phi}$ is nowhere vanishing on $\overline{\mathbb{D}}$, then $T_{\phi}$ is surjective with n-dimensional kernel.
If $\tilde{\phi}$ has a zero $w$ on $\mathbb{D}$ such that 
$$res_w(\widetilde{\frac{\partial\phi}{\partial z}}\tilde{\phi}^{-1})\notin\mathbb{Z}_{\geq n+1},$$
then  the kernel of $T_{\phi}$ is at most $n-1$-dimensional.
If $\phi\in H^{\infty}(\overline{\mathbb{D}})[\bar{z}]$ and $\tilde{\phi}$ has a single zero $w$ on $\mathbb{D}$
 and no zeroes on $\partial \mathbb{D},$ then $T_{\phi}$
is onto if and only if the above condition holds, in which case $\ker T_{\phi}$ is $(n-1)$-dimensional.
\end{theorem}



Our next result provides invertibility criterion for Toeplitz operators $T_{\phi}$
where $\phi$ is an $n$-th order harmonic function such that $\tilde{\phi}$ has
a zero with multiplicity $n$ in $\mathbb{D}.$

\begin{theorem}\label{order}

Let $0\neq w\in \mathbb{D}, \psi\in H^{\infty}(\mathbb{D})$ and
 $\phi=(z-w)^n\psi(z)+(\bar{z}-{w}^{-1})^n.$ Then
Toeplitz operator $T_{\phi}$ is injective if the following equation has no roots in $\mathbb{Z}_{\geq 0}:$
$$\prod_{i=1}^n(\lambda+i+1)+\psi(w)(-w)^n(n!+\sum_{i=1}^n\frac{n!^2}{i!^2(n-i)!}\lambda\cdots(\lambda-i+1))=0.$$
Moreover,  $\dim \ker(T_{\phi})$ equals to the number of distinct roots of the above equation  in $\mathbb{Z}_{+}$ 
 if  $1+(-z)^n\psi(z)$ has no zeroes on $\overline{\mathbb{D}}.$
In particular, if $\psi(w)(-w)^n\notin \mathbb{Q}_{{<-1}}$ and $(-z)^n\psi(z)$ is not equal to $-1$ on $\overline{\mathbb{D}}$,
then $T_{\phi}$ is invertible. Thus if $\psi\in H(\overline{\mathbb{D}})$ such that $(-z)^n\psi(z)\neq -1$ on $\overline{\mathbb{D}}$
and $(-z)^n\psi(z)(\mathbb{D})$ is a starlike domain around 0, then $T_{\phi}$ is invertible for any $w.$
\end{theorem}

Finally, let us recall the following version of  a question of Douglas about invertibility
of Toeplitz operators for harmonic symbols in the Bergman space setting.

\begin{Question}
Let $\phi(z)\in C(\overline{\mathbb{D}})$ be a nowhere vanishing harmonic function in a neighbourhood of $\overline{\mathbb{D}},$
then is $T_{\phi}$ invertible? More generally, is the spectrum of $T_{\phi}$ a subset of $\overline{\phi(\mathbb{D})}?$

\end{Question}
Remark that while given a nowhere vanishing harmonic $\phi$ as above, then the winding number of
 $\phi|_{\partial D}$ around 0 is 0 (hence $T_{\phi}$ has index 0), inverse of this statement is certainly
 not true: There are examples of vanishing harmonic functions on $\mathbb{D}$ with winding number on the boundary
 around 0 being 0 (see \cite{ZZ}). To the best of our knowledge there are no negative answers
 known in the existing literature to the above version of Douglas's question.
We will show in Corollary \ref{example} that Douglas's question has an affirmative answer for harmonic 
polynomials of the form $\phi=\bar{z}+f(z)$, where $f(z)$ is a quadratic
polynomial.  Note however that  even for quadratic
$f(z)$ $\phi(\mathbb{D})$  need not  be a subset of the spectrum of $T_{\phi}$ as shown in [\cite{ZZ}, Theorem 4.1].
For a linear $f(z)$ the spectrum of $T_{\phi}$ does equal to $\phi(\overline{\mathbb{D}})$ [\cite{ZZ}, Theorem 3.1]. 

\section{The differential operator $D_{\phi}$}

Recall the following well known formula 
$$T_{\bar{z}^k}(z^n)=
 \begin{cases} 
      0 & k>n \\
      \frac{n-k+1}{n+1}z^{n-k} & n\geq k. \\
      
   \end{cases}
$$

This easily implies that for any polynomial $f\in \mathbb{C}[z]$, we have
$$T_{\bar{z}^k}(f)=\prod_{i=2}^{k+1}(zD+i)^{-1}D^k(f),$$
Where $D: H(\mathbb{D})\to H(\mathbb{D})$ denotes the differentiation operator, and
$zD+i: H(\mathbb{D})\to H(\mathbb{D}), i>0$ are understood as invertible differential operators.

In particular, $T_{\bar{z}}=(zD+2)^{-1}D$ which is equivalent
to the following formula from Sundberg-Zheng \cite{SZ}
$$T_{\bar{z}}f=\frac{1}{z^2}\int_0^zwf'(w)dw.$$

\begin{proof}[Proof of Lemma \ref{key}]
Since $\Lambda=\prod_{i=2}^{n+1}(zD+i):H(\mathbb{D})\to H(\mathbb{D})$ is an injective linear operator,
it suffices to check that $\Lambda(T_h)(f)=\Lambda(g).$ Hence we need to show that
$$\Lambda(T_{a_i\bar z^i})(f)=\prod_{k=i+2}^{n+1}(zD+i)D^i(a_i(z)f)$$
It suffices to check this equality for $f=z^m, m\geq 0.$ But this is immediate from the above
discussion.

\end{proof}



Next we will compute the first two leading terms of $D_{\phi},$ i.e. coefficients in front of $D^n, D^{n-1}.$
Clearly the leading term of $D_{\phi}$ is $\tilde{\phi}D^n=\sum_{i=0}^na_i(z)z^{n-i}D^n.$
Recall that the following commutator relation holds in the ring of differential operators
$$Dg(z)-g(z)D=g'(z),\quad g(z)\in H(D).$$
Using this relation we easily obtain the following expansion in terms of powers of $D$
$$\prod_{k=1}^m(zD+b_k)=z^mD^m+(m(m-1)/2+\sum b_k)z^{m-1}D^{m-1}+\cdots, b_k\in H(\mathbb{D}).$$
Thus 
$$\prod_{k=1}^m(zD+b_k)\psi=\psi z^mD^m+((m(m-1)/2+\sum b_k)z^{m-1}\psi+mz^m\psi')D^{m-1}+\cdots, \psi\in H(\mathbb{D}).$$
Our differential operator is $D_{\phi}$ is
$$\prod_{i=2}^{n+1}(zD+i)a_0(z)+D^na_n(z)+(zD+n+1)D^{n-1}a_{n-1}(z)+\cdots+\prod_{i=3}^{n+1}(zD+i)D a_1(z).$$
Therefore the coefficient in front of $D^{n-1}$ is 
$$n(n+1)a_0(z) z^{n-1}+nz^na_0'(z)+na_n'(z)+(a_{n-k-1}(z)\sum_{k<n-1}(n+1)(k+1)+na_{n-k}'(z))z^{k}.$$
Which may be written as
$$(n+1)(\sum_{k=0}^n a_kz^{n-k})'+n(\sum_{k=0}^n a_k'z^{n-k}).$$

\noindent Which is equal to 
$$(n+1)\tilde{\phi}'-\sum a_k'z^{n-k}=(n+1)\tilde{\phi}'-\widetilde{\frac{\partial\phi}{\partial z}}.$$
To summarize we have 
$$D_{\phi}=\tilde{\phi}D^n+((n+1)\tilde{\phi}'-\widetilde{\frac{\partial\phi}{\partial z}})D^{n-1}+\cdots.$$


 
 
 .





\begin{proof}[Proof of Theorems \ref{res}, \ref{order}]
By Lemma \ref{key}, we will need to analyse the dimension of the space of solutions
of the differential equation $D_{\phi}(y)=0$ for $y\in A^2(\mathbb{D}).$
Since it is an $n$-th order homogeneous equation, its space of solutions is at most n-dimensional,
thus $\dim\ker T_{\phi}\leq n$ for any $n$-th order polyanalytic function $\phi$.
Now suppose that $\phi \in H(\overline{\mathbb{D}})[\bar{z}]$ is of order $n$ such that
$\tilde{\phi}$ has no zeroes on $\overline{\mathbb{D}}.$
Thus the index of $T_{\phi}$ is $n$, on the other hand $\dim \ker T_{\phi}\leq n.$
Therefore $\dim \ker T_{\phi}=n$ and $\dim \coker T_{\phi}=0.$

The rest of the proof will proceed by observing that the differential equation $D_{\phi}(y)=0$  has regular singularity
at $w,$ and then applying the Frobenius method provides the corresponding indicial equation
in $\lambda$, where $\lambda$ is the smallest nonzero power of $(z-w)$ appearing in the Taylor expansion
of a nontrivial solution $y$ at $w.$ Indeed, recall that  an $n$-th order differential
equation $\sum a_i(z)y^{(n)}=0$ is said to have regular singularity at w if $(z-w)^{-i}a_i(z)$ is holomorphic in a neighbourhood
of $w.$ For such an equation let $a_i$ be the value of $(z-w)^{-i}a_i(z)$ at $w.$
We will refer to $\sum_i a_iD^i$ as the essential part at $w$ of the differential operator
$\sum_ia_i(z)D^i.$
Then the indicial equation of the above differential equation is $\sum_{i=0}^n\lambda\cdots(\lambda-i+1)a_i=0.$
Recall that the dimension of the space of holomorphic solutions
around $w$ equals to the number of distinct roots of the indicial equation in $\mathbb{Z}_{+}.$
Moreover, a classical theorem of Fuchs' \cite{H} asserts that if $a_i(z)$ are holomorphic in a neighbourhood of $\overline{\mathbb{D}}$
then for each such root $\lambda\in \mathbb{Z}_{+}$ there is a holomorphic solution around
$\overline{\mathbb{D}}$ with order of vanishing at $w$ equalling $\lambda.$ 

Now in the setting of Theorem \ref{res}, it follows that the equation $(z-w)^{n-1}D_{\phi}(y)=0$ has
a regular singularity at $w,$ and the corresponding indicial equation is
$$\lambda(\lambda-1)\cdots(\lambda-n+2)(\lambda-n+1+res_w\frac{(n+1)\tilde{\phi}'-\widetilde{\frac{\partial\phi}{\partial z}}}{\tilde{\phi}})=0,$$
which gives $\lambda=res_w({\tilde{\phi}^{-1}}\widetilde{\frac{\partial\phi}{\partial z}})-2$ and $\lambda=0,1,\cdots, n-2.$
Thus we are done by Fuchs' theorem.

Now we will show Theorem \ref{order}. We want to show that the differential operator $D_{\phi}$ for  $\phi(z)=(z-w)^n\psi(z)+(\bar{z}-w^{-1})^n$
has regular singularity at $w$, and then compute its essential part. Then we will obtain the desired indicial  equation
by evaluating the essential part on $(z-w)^{\lambda}$ and setting it to equal 0.
Recall that 
$$(zD+2)(T_{\bar{z}}-w^{-1})=-w^{-1}((z-w)D+2),\quad (zD+i)T_{\bar{z}}=T_{\bar{z}}(zD+i-1).$$
Hence $$(zD+i)(T_{\bar{z}}-w^{-1})=(T_{\bar{z}}-w^{-1})(zD+i-1)-w^{-1}.$$
Now we can obtain the following recursive equality
$$(zD+n+1)(T_{\bar{z}}-w^{-1})^n=(T_{\bar{z}}-w^{-1})(zD+n)(T_{\bar{z}}-w^{-1})^{n-1}-w^{-1}(T_{\bar{z}}-w^{-1})^{n-1}.$$
which yields 
$$(zD+n+1)(T_{\bar{z}}-w^{-1})^n=-(n-1)w^{-1}(T_{\bar{z}}-w^{-1})^{n-1}+(T_{\bar{z}}-w^{-1})^{n-1}(-w^{-1}((z-w)D+2))$$
Hence we have the following recursive formula  (put $L=((z-w)D+2)$ for brevity)
$$\prod_{i=2}^{n+1}(zD+i)(T_{\bar{z}}-w^{-1})^n=(-w^{-1})\prod_{i=2}^{n}(zD+i)(T_{\bar{z}}-w^{-1})^{n-1}(L+(n-1)).$$
Finally, we have
$$\prod_{i=2}^{n+1}(zD+i)(T_{\bar{z}}-w^{-1})^n=(-w^{-1})^n\prod_{i=1}^n (L+(i-1)).$$
The latter has regular singularity at $w$, and evaluated on $(z-w)^{\lambda}$ gives
$$(-1)^nw^{-n}\prod_{i=1}^n(\lambda+i+1).$$
Next we will compute the essential part at $w$ of the differential operator $\prod_{i=2}^{n+1}(zD+i)(z-w)^n\psi(z).$
Suffices to compute this for $\prod_{i=2}^{n+1}(zD+i)(z-w)^n.$
On the other hand the essential part at $w$ of differential operator $D^i(z-w)^n$ is $0$ unless $n=i.$
So the desired essential part is that of $(zD)^n(z-w)^n,$ which is equal to
the essential part of $w^nD^n(z-w)^n.$ Now recall a well-known equality in the algebra of differential operators
$$D^n (z-w)^n=(z-w)^nD^n+\sum_{i=0}^{n-1}\frac{n!^2}{i!^2(n-i)!}(z-w)^iD^i$$
Hence the essential part is $\psi(w)\sum_{i=0}^n\frac{n!^2}{i!^2(n-i)!}D^i.$
Finally, we obtain the sought after indicial equation is 
$$\prod_{i=1}^n(\lambda+i+1)+\psi(w)(-w)^n\sum_{i=0}^n\frac{n!^2}{i!^2(n-i)!}\lambda\cdots(\lambda-i+1)=0.$$
Now the desired result follows by Fuchs' theorem just as in the proof of Theorem \ref{res}.

Finally, it suffices to see that the above equation has no solutions in $\mathbb{N}$ for $(-w)^n\psi(w)\notin \mathbb{Q}_{<-1}.$
The letter follows from an easy fact that 
$$\prod_{i=1}^n(\lambda+i+1)>\sum_{i=0}^n\frac{n!^2}{i!^2(n-i)!}\lambda\cdots(\lambda-i+1)$$
for all $n, \lambda\in \mathbb{N}.$


\end{proof}

\begin{remark}
Let $\phi$ as above be such that the winding number around 0 of $\phi(\partial \mathbb{D})$ is 0.
Then for such generic $\phi$, the corresponding Toeplitz operator is invertible. 
Let $w_1,\cdots, w_n$ be zeros of $\tilde{\phi}$ on $\mathbb{D}.$ 
It follows that $T_{\phi}$ is not invertible if and only if
equation $D_{\phi}f=0$ has a nontrivial solution in $A^2(\mathbb{D}).$ 
Let $y'=Ay$ be the matrix form of this equation.
For generic such $\phi$ it follows that this equation
has regular singularities at $w_i.$ 
Put $A_i=res_{w_i}A.$ Thus generically, distinct eigenvalues of $A_i$ do not differ by integers
Let $ M_1,\cdots, M_n$ be the monodromy matrices around $w_1,\cdots, w_n$ respectively.
Then $M_i$ is conjugate to $\exp(2\pi iA_i)$, hence it has an eigenvalue 1 with multiplicity $n-1.$
Then existence of such a solution implies that matrices $M_1,\cdots, M_n$ have a simultaneous eigenvector
with eigenvalue 1. But generically this does not hold.

\end{remark}




Let us  explicitly write down the differential equation $D_{\phi}y=0$ of $n=2.$ For computational simplicity we will
consider harmonic functions 
$$\phi=a_2\bar{z}^2+a_1\bar{z}+a_0(z),\quad a_1, a_2\in \mathbb{C}, a_0(z)\in H(D).$$
So $\tilde{\phi}=a_2+a_1z+a_0z^2.$ 
 The corresponding differential equations is
$$(6+z^2D^2+6zD)(a_0 y)+a_2(D)^2y+a_1(zD+3)D y=0,$$
which simplifies to
$$\tilde{\phi}y''+(3(\tilde{\phi})'-z^2a_0')y'+(3(\tilde{\phi})''-2z^2a_0'')y=0.$$

  We will end  by explicitly working out invertibility criteria for $T_{\phi}$  for certain relatively simple
 harmonic functions $\phi$. 
 

\begin{cor}\label{example}

Let $\phi(z)=az^n(z-w)^m+(\bar{z}-w^{-1})^m$ with $a\in\mathbb{C}, w\in \mathbb{D}\setminus \lbrace 0\rbrace$ 
and $n,m \in\mathbb{N}.$ Then $T_{\phi}$ is invertible if and only if $|a|<|w|^{-n}.$
Let $\phi(z)=\bar{z}+a+bz+cz^2$ with $a,b, c\in\mathbb{C}.$ Then
the spectrum of $T_{\phi}$ is a subset of $\overline{\phi(\mathbb{D})}.$
\end{cor}
\begin{proof}
If $|a|>1$ then the index of $T_{\phi}$ is nonzero. Indeed,
as $\tilde{\phi}=w^{-n}(z-w)^m(1+a(-w)^nz^{n+m}),$ it has $m$ roots in $\mathbb{D}$ (hence $T_{\phi}$ has index 0)
 if and only if  $|a|<|w|^{-n}.$
Now suppose that  $|a|<|w|^{-n}.$ Then $T_{\phi}$ is invertible by Theorem \ref{order}.

Finally let  $\phi=\bar{z}+a+bz+cz^2.$ 
It suffices to show that $T_{\phi}$ is invertible if  $T_{\phi}$ has index . Thus $\tilde{\phi}=1+z(a+bz+cz^2)$ has exactly
one zero $w\in \mathbb{D}.$ 
Then we may write  $\phi=-w^{-1}(\bar{z}-w^{-1})+(z-w)\psi$, where $\psi$ is a linear function.
Then $(-z)\psi(z)(\mathbb{D})$ is a starlike domain  around 0. Hence by Theorem \ref{order} $T_{\phi}$
is invertible.
\end{proof}





\begin{acknowledgement} 
I am  grateful to Z.Cuckovic and T.Le for their interest in this work
which led to significant improvements over the previous version of the paper.

\end{acknowledgement}

\end{document}